\documentclass{amsart}
\usepackage{amsmath, amsthm, amssymb}
\usepackage[mathscr]{eucal}
\usepackage{enumerate}
\usepackage{enumitem}
\usepackage{tikz}
\usepackage{graphicx}
\newcommand{\R}{{\mathbf R}}

\newcommand{\F}{\mathcal{F}}

\newcommand{\norm}[1]{{\left\| #1 \right\|}}

\newcommand{\abs}[1]{{\left| #1 \right|}}

\newcommand{\allow}{\operatorname{Allow}}
\newcommand{\pat}{\operatorname{Pat}}

\newcommand{\Pat}{\operatorname{Pat}}
\newcommand{\Allow}{\operatorname{Allow}}
\renewcommand{\S}{\mathcal{S}}

\newtheorem{prop}{Proposition}[section]
\newtheorem{thm}[prop]{Theorem}
\newtheorem{cor}[prop]{Corollary}
\newtheorem{lem}[prop]{Lemma}
\newtheorem{conj}{Conjecture}

\theoremstyle{definition}
\newtheorem{defn}[prop]{Definition}

\newtheorem{exmp}[prop]{Example}
\newtheorem{rem}[prop]{Remark}

\newlist{thmenum}{enumerate}{10}
\setlist[thmenum,1]{label=\textnormal{(\alph*)}}
\setlist[thmenum,2]{label=\textnormal{(\roman*)}}

   \allowdisplaybreaks
   
\title{Allowed patterns of symmetric tent maps via commuter functions}
\author{Kassie Archer and Scott M. LaLonde}
\date{\today}
   
\begin{document}
\maketitle

\begin{abstract}
	We introduce a new technique to study pattern avoidance in dynamical systems, namely the use of a commuter function between non-conjugate dynamical systems. We 				investigate the properties of such a commuter function, specifically $h : [0,1] \to [0,1]$ satisfying $T_1 \circ h = h \circ T_\mu$, where $T_\mu$ denotes a symmetric
	tent map of height $\mu$. We make use of this commuter function to prove strict inclusion of the set of allowed patterns of $T_\mu$ in the set of allowed patterns of $T_1$.
\end{abstract}


\section{Introduction}

Let $\S_n$ denote the set of permutations of $[n] = \{1, 2, \ldots, n\}$. We always write permutations in one-line notation: if $\pi \in \S_n$, we write
\[
	\pi=\pi_1\pi_2 \ldots \pi_n.
\]
Given a one-dimensional discrete dynamical system $f: [0,1]\to[0,1]$ and a positive integer $n$, we can associate permutations of length $n$ to certain points of $[0,1]$
as follows. Let $x \in [0,1]$, and assume $x$ is not a $k$-periodic point for any $k < n$. Define $\Pat(x, f, n)$ to be the permutation $\pi_1 \pi_2 \cdots \pi_n \in
\S_n$ whose entries are in the same relative order as the first $n$ elements of the orbit of $x$ with respect to $f$. That is, $\pi_1, \pi_2, \ldots, \pi_n$ are in the 
same relative order as
 \[
	x, f(x), f^2(x), \ldots, f^{n-1}(x).
 \]
 We call $\Pat(x, f, n)$ the \emph{ordinal pattern}, or simply \emph{pattern}, of $x$ with respect to $f$ of length $n$. 
 
 \begin{exmp}
 The pattern of $0.23$ with respect to the standard tent map, $T$, of length 5 is the permutation of length 5 in the same relative order as
 \[
 	0.23, T(0.23), T^2(0.23), T^3(0.23), T^4(0.23) 
 \] 
which when evaluated, gives:
\[
0.23, 0.46, 0.92, 0.16, 0.32.
\] 
Therefore, the pattern is $\Pat(0.23, T, 5) = 24513$.
\end{exmp} 

We call the set of all such patterns realized by elements of $[0,1]$ the \emph{allowed patterns} of $f$, denoted by $\Allow(f)$. The set of allowed patterns of $f$ of length 
$n$ is denoted by $\Allow_n(f)$.  Any permutation which is not realized as an allowed pattern of $f$ is called a \emph{forbidden pattern} of $f$. For example, the 
permutation $321\in \S_3$ is a forbidden pattern of $T$ since there is no $x\in [0,1]$ for which the sequence $x, T(x), T^2(x)$ is in decreasing order. 

The allowed and forbidden patterns of many maps from dynamical systems have been studied during the last several years, including the left shift on words \cite{Elishifts, Makarov}, signed shifts on the unit 
interval \cite{Amigosigned, AEK, ArcAllow, AE}, beta shifts \cite{Elibeta}, negative beta-shifts \cite{EM} and the logistic maps \cite{Eliu}.

It is known that for a piecewise monotone map $f: [0,1]\to[0,1]$, the size of $\Allow_n(f)$ grows at most exponentially \cite{BKP}, and thus $f$ has forbidden patterns (since the size 
of $\S_n$ grows super-exponentially). Forbidden patterns of such maps allow one to distinguish a random time series from a deterministic one \cite{AZS, AZS08, AZS2}. This occurs 
since most patterns are forbidden in a deterministic time series, while a random time series eventually contains all patterns. In addition, the size of $\abs{\Allow_n(f)}$ for a given $f$ is 
known to be directly related to the topological entropy of $f$, which measures the complexity of $f$ \cite{BKP}. Furthermore,  these ideas have also led to purely combinatorial results 
in the study of permutations \cite{AE, Elishifts, Elicyc}.

For the reasons described above, studying the allowed and forbidden patterns of a given map $f$ presents an interesting problem. In \cite{ArcAllow}, the patterns 
realized by the standard tent map, $T$, are characterized and partially enumerated. Here, we study the relationship between the allowed and forbidden patterns 
of an arbitrary symmetric tent map and those of the standard one. Given $\mu > 0$, we define the \emph{symmetric tent map} of height $\mu$ to be the piecewise linear 
function
\[
	T_\mu(x) = \begin{cases}
		2\mu x & \text{ if } 0 \leq x \leq 1/2 \\
		2\mu(1-x) & \text{ if } 1/2 < x \leq 1.
	\end{cases}
\]
This gives us a one-parameter family of discrete dynamical systems on the interval $[0,1]$. The tent maps $T=T_1$ and $T_{3/4}$ are depicted below.

\begin{figure}[h]
\label{fig:tent}
	\includegraphics[width=6cm]{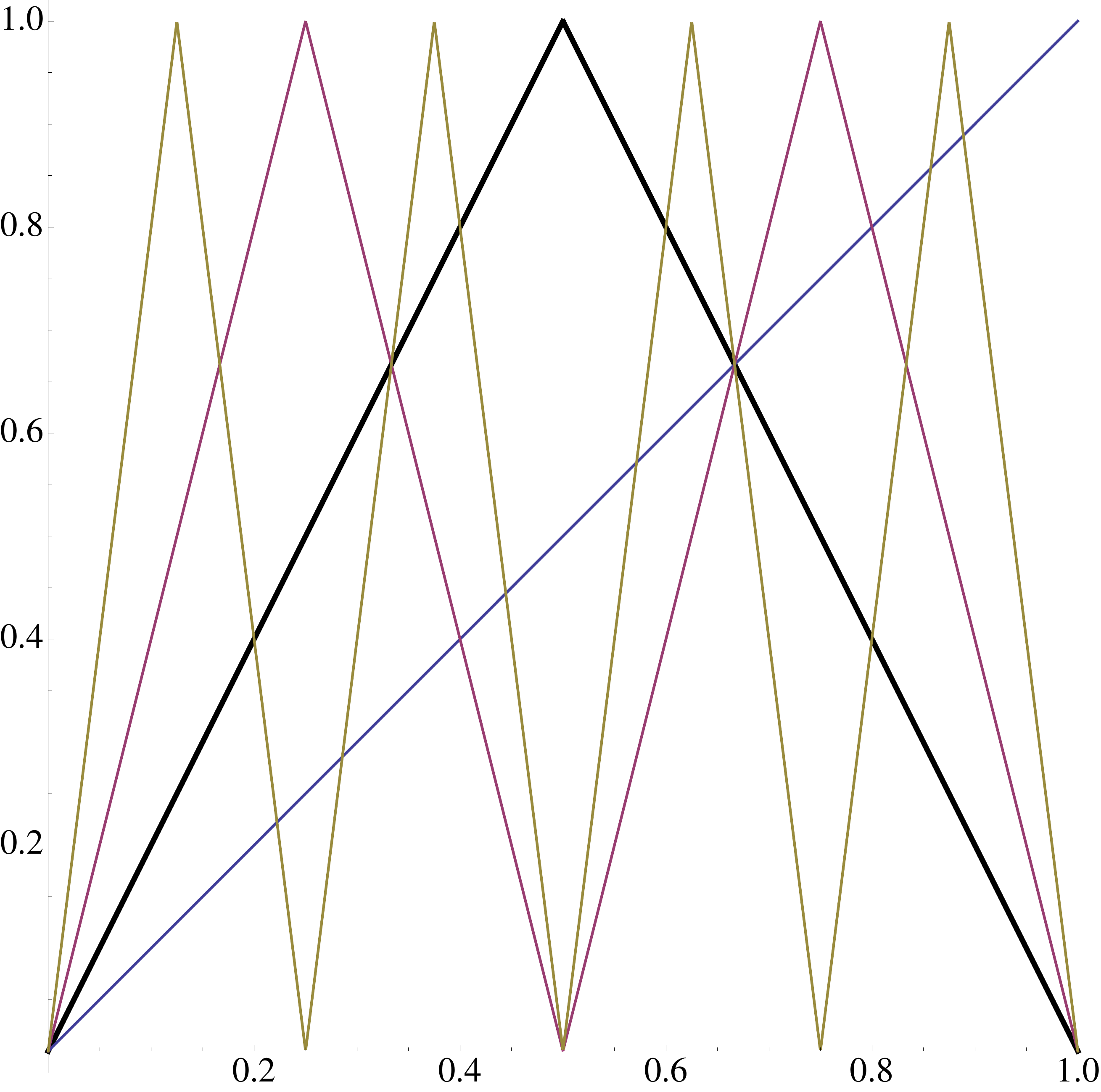}
	\hspace{2mm}
	\includegraphics[width=6cm]{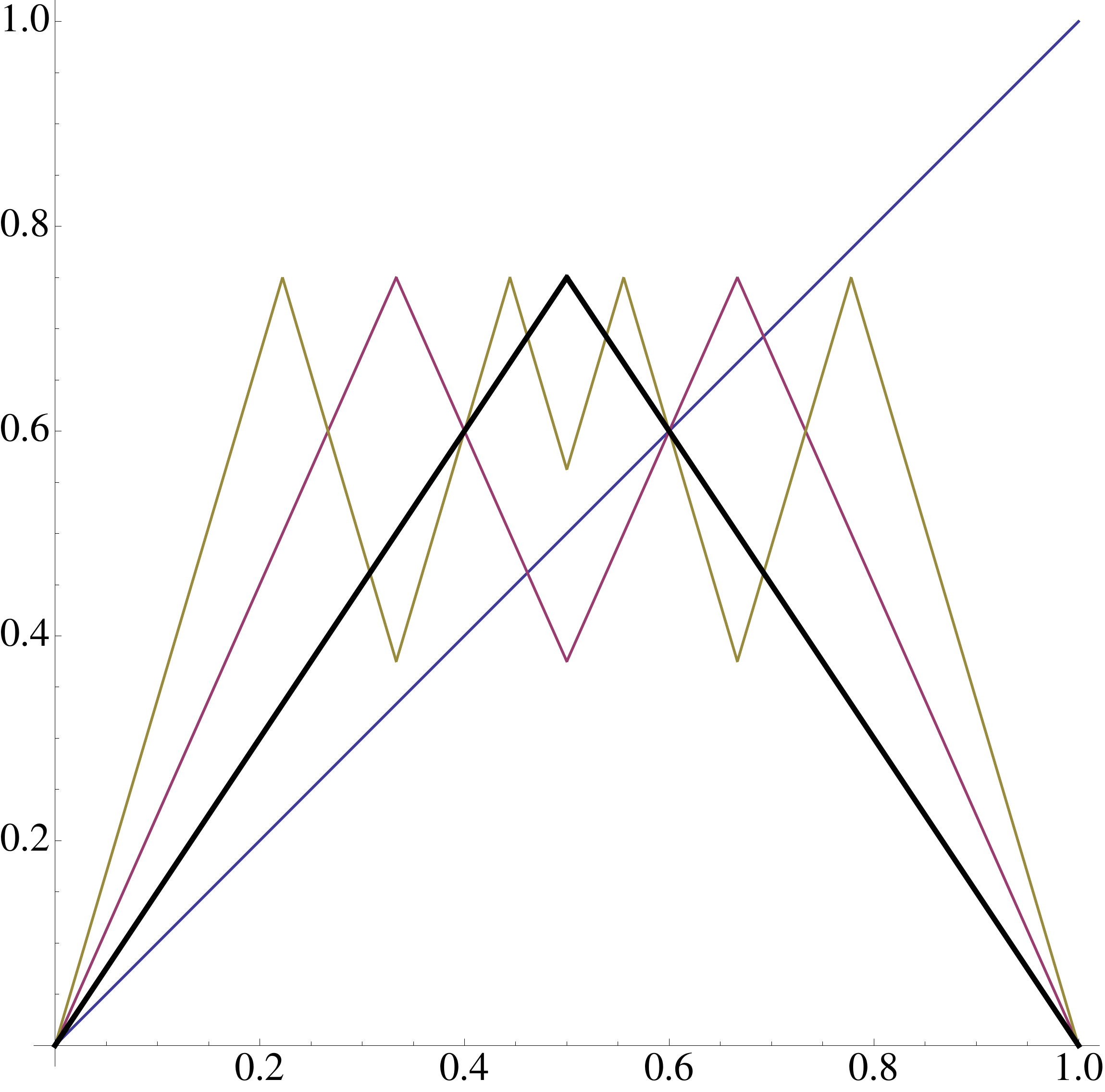}
	\caption{The first three iterates of the standard tent map $T$ (left) and $T_{3/4}$ (right), together with the line $y=x$. The tent maps themselves are depicted in bold.
	Notice that for $x$ near $1/2$, $\pat(x, T, 4) = 3412$. On the other hand, $\pat(x, T_{3/4}, 4) = 2413$ for $x$ sufficiently close to $1/2$. In fact, one can easily observe
	from the figure that the pattern $3412$ is forbidden for $T_{3/4}$.}
\end{figure}

\noindent We refer to the special case $T_1$ as the \emph{standard} or \emph{full tent map}, and we denote it simply by $T$. We also restrict our investigation to the 
situation where $1/2 < \mu \leq 1$, since the dynamics of $T_\mu$ are fairly degenerate when $\mu \leq 1/2$. For example, $T_\mu$ has an attracting fixed point
if $\mu < 1/2$. Also, $T_{1/2}$ has a continuum of fixed points.

As mentioned above, we aim to analyze the relationship between the allowed patterns of a tent map $T_\mu$ for $\mu < 1$ and the allowed patterns of the 
standard tent map $T$. One can already see from Figure \ref{fig:tent} that $T_{3/4}$ has less complex dynamics than $T$, and fewer allowed patterns. In particular,
$2341, 3412, 3124 \in \Allow(T)$, but these patterns are all forbidden for $T_{3/4}$. On the other hand, it is straightforward to check that all patterns in 
$\Allow_4(T_{3/4})$ are realized by $T$. It thus seems plausible to conjecture that $\Allow(T_\mu) \subseteq \Allow(T)$ whenever $1/2 < \mu \leq 1$.

One of the main results of this paper is a proof of the above conjecture. We prove it by constructing a strictly increasing (but not necessarily surjective or 
continuous) map $h_\mu : [0,1] \to [0,1]$ satisfying 
\begin{equation}
\label{eq:commutation}
	T \circ h_\mu = h_\mu \circ T_\mu.
\end{equation}
We will often refer to \eqref{eq:commutation} as the \emph{commutation relation}. Functions of this type have been studied in \cite{bollt-skufca}, where they 
are called \emph{commuters}. In that paper, the authors describe methods for constructing commuters, and they develop a particularly nice iterative process for 
building a conjugacy between an asymmetric tent map and a symmetric one. These functions usually look quite bizarre, since they exhibit a certain kind of 
self-similar structure by construction.

The iterative process used in \cite{bollt-skufca} to construct conjugacies can be easily adapted to build a non-homeomorphic commuter between $T_\mu$ and $T$.
We construct such a function and analyze its properties; in particular, we show that the points of discontinuity are dense 
in $[0,1]$, and that $h$ is strictly increasing. We investigate the range of $h$ (which we believe to be a Cantor-like set), and we then study the 
implications for patterns realized by the tent maps $T_\mu$ and $T$.

In Section \ref{secCommuter}, we define commuter functions and prove properties of the commuter function between tent maps. In Section \ref{secRange}, we further investigate the range of the commuter functions. In Section \ref{secAllowed}, we discuss the implications these results have for the allowed and forbidden patterns of $T_\mu$. Finally, in Section \ref{secOpen}, we discuss a few conjectures.

\section{Commuter Functions}\label{secCommuter}

Our stated goal is to study the relationship between the allowed patterns of two different tent maps. To shed some light on this question, we begin with a simpler
one. When do two dynamical systems $f, g : [0, 1] \to [0, 1]$ have the same allowed patterns? This question is tantamount to asking that $f$ and $g$ have the
``same'' dynamics. Put more precisely, in order for two dynamical systems $f, g : [0,1] \to [0,1]$ to have the same allowed patterns, it is necessary that they are 
\emph{conjugate}, meaning there is a homeomorphism $h : [0,1] \to [0,1]$ such that 
\[
	f = h^{-1} \circ g \circ h.
\]
Since we are dealing with maps on the unit interval, any such homeomorphism $h$ must be continuous and either strictly increasing or strictly decreasing. It is 
straightforward to show that if $h$ is strictly increasing (i.e., it is an order-preserving conjugacy), then $f$ and $g$ have the same allowed patterns. 

\begin{thm}
\label{thm:conjugate_patterns}
	Let $f, g : [0, 1] \to [0, 1]$ be two dynamical systems, and suppose there is a strictly increasing surjection $h : [0, 1] \to [0, 1]$ satisfying
	$f = h^{-1} \circ g \circ h$. Then $\Allow(f) = \Allow(g)$.
\end{thm}
\begin{proof}
	Let $\pi \in \Allow(f)$ be a pattern of length $n$, and choose $x \in [0, 1]$ such that $\Pat(x, f, n) = \pi$. That is,
	\[
		x, f(x), f^2(x), \ldots, f^{n-1}(x)
	\]
	is in the same relative order as $\pi_1, \pi_2, \ldots, \pi_n$. Since $h$ is strictly increasing,
	\begin{equation}
	\label{eq:order}
		h(x), h(f(x)), h(f^2(x)), \ldots, h(f^{n-1}(x))
	\end{equation}
	is also in the same relative order. But we have $h \circ f = g \circ h$ by assumption, so the points in \eqref{eq:order} can be rewritten as
	\[
		h(x), g(h(x)), g^2(h(x)), \ldots, g^{n-1}(h(x)).
	\]
	This means that $\pi = \Pat(h(x), g, n)$, so $\pi \in \Allow(g)$. Hence $\Allow(f) \subseteq \Allow(g)$. The same argument shows that if $\pi \in \Allow(g)$
	is realized at a point $x \in [0,1]$, then $\pi$ is realized by $f$ at $h^{-1}(x)$. Thus $\Allow(f) = \Allow(g)$.
\end{proof}

Unfortunately, $T_\mu$ and $T$ are not conjugate if $\mu \neq 1$. (An easy way to see this is that the two maps have different topological entropies.) Therefore,
we replace the notion of conjugacy with the commutation relation defined in the introduction, and seek a function $h_\mu : [0,1] \to [0,1]$ satisfying
\[
	T \circ h_\mu = h_\mu \circ T_\mu.
\]

\begin{defn} 
	Let $f, g : [0,1] \to [0,1]$ be dynamical systems. We say that a function $h : [0,1] \to [0,1]$ is a \emph{commuter} for $f$ and $g$ if
	\[
		f \circ h = h \circ g.	
	\]
\end{defn}

As mentioned in the introduction, commuters have been studied in \cite{bollt-skufca}. The authors also exploit the commutation relation to build conjugacies that are
otherwise hard to write down. For example, they present an iterative process for constructing a conjugacy between a skew tent map and a symmetric one. It has been
observed in \cite{zheng} and \cite{lalonde} that a similar procedure can be used to construct commuters between non-conjugate dynamical systems in special cases.

In general, we can say something about the relationship between the set of allowed patterns of two maps $f$ and $g$ if there is a commuter which is order-preserving (i.e. increasing, when $f$ and $g$ are maps on the unit interval). 
\begin{thm}
\label{thm:semiconjugate_patterns}
	Let $f, g : [0, 1] \to [0, 1]$ be two dynamical systems, and suppose there is a strictly increasing function $h : [0, 1] \to [0, 1]$ satisfying
	$f \circ h = h \circ g$. Then $\Allow_n(g) \subseteq \Allow_n(f)$ for all $n\geq 1$.
\end{thm}
\begin{proof}
The argument is the same as in the proof of Theorem~\ref{thm:conjugate_patterns}. Suppose $\pi \in \allow_n(g)$ for some $n$, and that $\pi$ is realized by $g$ at $x$. In other words,
	\[
		x, g(x), g^2(x), \ldots, g^{n-1}(x)
	\]
	is in the same relative order as
	\[
		\pi_1, \pi_2, \ldots, \pi_n.
	\]
	Since the commuter $h$ is strictly increasing on the unit interval, it is order-preserving, so
	\begin{equation}
	\label{eq:pattern}
		h(x), h(g(x)), h(g^2(x)), \ldots, h(g^{n-1}(x))
	\end{equation}
	is in the same relative order as the entries of $\pi$. But we know that $h(g^k(x)) = f^k(h(x))$ for all $k$, so \eqref{eq:pattern} is just the pattern 
	of $f$ at $h(x)$:
	\[
		h(x), f(h(x)), f^2(h(x)), \ldots, f^{n-1}(h(x)).
	\]
	Thus $\pi$ is realized by $f$ at $h(x)$, so $\pi \in \allow(f)$.
\end{proof}

In our setting, we would like to find a function $h_\mu$ satisfying the commutation relation with $f = T$ and $g = T_\mu$ for a given value of $\mu$. To do so, we 
modify the construction from Section II.B of \cite{bollt-skufca}. The details are more or less the same, but we still attempt to provide a self-contained treatment of 
the construction. Note first that the commutation relation \eqref{eq:commutation} just says that 
\[
	T(h_\mu(x)) = h_\mu(T_\mu(x))
\]
for all $x \in [0,1]$. If $x \in [0,1/2]$, this equation becomes
\begin{equation}
\label{eq:one}
	T(h_\mu(x)) = h_\mu(2\mu x),
\end{equation}
while for $x \in (1/2, 1]$ we have
\begin{equation}
\label{eq:two}
	T(h_\mu(x)) = h_\mu(2\mu(1-x)).
\end{equation}
Even though $h_\mu$ is not a conjugacy, it should preserve the monotone intervals of $T$ and $T_\mu$ if it is to give us any meaningful information about the dynamics
and allowed patterns. Therefore, we require that $h_\mu([0,1/2]) \subseteq [0,1/2]$ and $h_\mu((1/2,1]) \subseteq(1/2,1]$. Under this assumption, \eqref{eq:one} becomes
\[
	2h_\mu(x) = h_\mu(2\mu x),
\]
or
\[
	h_\mu(x) = \tfrac{1}{2} h_\mu(2\mu x).
\]
On the other hand, \eqref{eq:two} yields
\[
	2(1-h_\mu(x)) = h_\mu(2\mu(1-x))
\]
or
\[
	h_\mu(x) = 1-\tfrac{1}{2}h_\mu(2\mu(1-x)).
\]
Therefore, $h_\mu$ is a commuter if it satisfies the functional equation
\renewcommand{\arraystretch}{1.75}
\begin{equation}
\label{eq:functional}
	h_\mu(x) = \left\{ \begin{array}{cc}
		\frac{1}{2} h_\mu(2\mu x) & \text{ if } 0 \leq x \leq 1/2 \\
		1-\frac{1}{2}h_\mu(2\mu(1-x)) & \text{ if } 1/2 < x \leq 1
	\end{array}
	\right.
\end{equation}

To show that such a function exists, we invoke the Contraction Mapping Theorem. Let $\mathcal{X} = B([0,1],\R)$ denote the space of bounded real-valued functions 
on $[0,1]$, which is a complete metric space under the norm $\norm{f}_\infty = \sup_{x \in [0, 1]} \abs{f(x)}$. Define an operator $M_\mu : \mathcal{X} \to \mathcal{X}$ by
\renewcommand{\arraystretch}{1.75}
\begin{equation}
\label{eq:operator}
	M_\mu f(x) = \left\{ \begin{array}{cc}
		\frac{1}{2} f(2\mu x) & \text{ if } 0 \leq x \leq 1/2 \\
		1-\frac{1}{2}f(2\mu (1-x)) & \text{ if } 1/2 < x \leq 1
	\end{array}
	\right.
\end{equation}
Note that $h_\mu$ is a solution to \eqref{eq:functional} precisely when it is a fixed point of $M_\mu$. Since $h_\mu$ should map the unit interval to itself, we are 
particularly interested in the restriction of $M_\mu$ to the closed subset
\[
	\mathcal{F} = \left\{ f \in \mathcal{X} \mid f : [0,1] \to [0,1] \right\}.
\]

\begin{lem}
	The operator $M_\mu$ maps $\mathcal{F}$ to itself. In particular, if $f \in \mathcal{F}$, then $M_\mu f$ maps $[0,1/2]$ to $[0,1/2]$ and $(1/2,1]$ to $[1/2,1]$.
\end{lem}
\begin{proof}
	Suppose $f \in \F$. If $x \in [0,1/2]$, then
	\[
		M_\mu f(x) = \tfrac{1}{2} f(2\mu x),
	\]
	which belongs to $[0,1/2]$ since $0 \leq f(2\mu x) \leq 1$. Thus $M_\mu f$ maps $[0,1/2]$ to $[0,1/2]$. Similarly, if $x \in (1/2,1]$, then
	\[
		M_\mu f(x) = 1-\tfrac{1}{2}f(2\mu (1-x))
	\]
	belongs to $[1/2,1]$. Consequently, $M_\mu f \in \F$.
\end{proof}


\begin{lem}
	The operator $M_\mu$ is contractive on $\F$.
\end{lem}
\begin{proof}
	Let $f, g \in \F$. Then we have
	\begin{align*}
		\sup_{x \in [0, \frac{1}{2}]} \abs{M_\mu f(x) - M_\mu g(x)} &= \sup_{x \in [0, \frac{1}{2}]} \abs{\tfrac{1}{2} f(2\mu x) - \tfrac{1}{2} g(2\mu x)} \\
		 	&= \tfrac{1}{2} \sup_{x \in [0, \frac{1}{2}]} \abs{f(2\mu x) - g(2\mu x)} \\
			&= \tfrac{1}{2} \sup_{x \in [0, \mu]} \abs{f(x) - g(x)} \\
			& \leq \tfrac{1}{2} \sup_{x \in [0, 1]} \abs{f(x) - g(x)}.
	\end{align*}
	Similarly,
	\begin{align*}
		\sup_{x \in (\frac{1}{2},1]} \abs{M_\mu f(x) - M_\mu g(x)} &= \sup_{x \in (\frac{1}{2},1]} \abs{(1-\tfrac{1}{2} f(2\mu (1-x))) - (1-\tfrac{1}{2} g(2\mu(1-x)))} \\
		 	&= \tfrac{1}{2} \sup_{x \in (\frac{1}{2},1]} \abs{g(2\mu(1-x)) - f(2\mu(1-x))} \\
			&= \tfrac{1}{2} \sup_{x \in [0, \mu)} \abs{g(x) - f(x)} \\
			& \leq \tfrac{1}{2} \sup_{x \in [0, 1]} \abs{f(x) - g(x)}.
	\end{align*}
	Thus $\norm{M_\mu f - M_\mu g}_\infty \leq \frac{1}{2} \norm{f-g}_\infty$ for all $f, g \in \F$, so $M_\mu$ is contractive.
\end{proof}

Since $\F$ is complete and $M_\mu$ is a contraction, the Contraction Mapping Theorem guarantees that $M_\mu$ has a unique fixed point $h_\mu \in \F$. But we have 
already observed that a fixed point for $M_\mu$ satisfies the functional equation \eqref{eq:functional}, and hence is the desired commuter. To summarize:

\begin{thm}
	The fixed point $h_\mu$ of the contraction $M_\mu$ satisfies the commutation relation $T \circ h_\mu = h_\mu \circ T_\mu$.
\end{thm}

\begin{rem}
	While $h_\mu$ is the unique fixed point of the contraction $M_\mu$ (hence the unique solution to the functional equation \eqref{eq:functional}),
	there are other commuters for the maps $T$ and $T_\mu$. We could have instead defined a contraction $M_\mu' : \mathcal{X} \to \mathcal{X}$ by
	\[
		M'_\mu f(x) = \left\{ \begin{array}{cc}
			\frac{1}{2} f(2\mu x) & \text{ if } 0 \leq x < 1/2 \\
			1-\frac{1}{2}f(2\mu (1-x)) & \text{ if } 1/2 \leq x \leq 1,
		\end{array} \right.
	\]
	which is equivalent to requiring that the commuter maps $[0, 1/2)$ to $[0, 1/2)$ and $[1/2, 1]$ to $[1/2, 1]$. This contraction yields a different
	commuter $h'_\mu$, though it agrees with $h_\mu$ everywhere except the points of discontinuity.
\end{rem}

\begin{rem}
	There is an extra advantage to our use of the Contraction Mapping Theorem. Since its proof is constructive, we obtain an iterative process for
	defining the fixed point $h_\mu$. If we start with any function $f_0 \in \F$ and define the sequence of functions
	\[
		f_{n} = M_\mu f_{n-1},
	\]
	then $f_n \to h_\mu$ uniformly. That is, we can define
	\[
		h_\mu = \lim_{n \to \infty} f_n.
	\]
	It is often useful to take either $f_0(x) = x$ or $f_0(x) = 1/2$. This construction also gives us an estimate for the speed of convergence. If $f_0 \in \F$, 
	then $\norm{f_0 - h_\mu}_\infty \leq 1$. Therefore,
	\[
		\norm{f_1 - h_\mu}_\infty = \norm{M_\mu f_0 - M_\mu h_\mu}_\infty \leq \tfrac{1}{2} \norm{f_0 - h_\mu} \leq \tfrac{1}{2}.
	\]
	Continuing inductively, we find that
	\[
		\norm{f_n - h_\mu}_\infty \leq \tfrac{1}{2^n}
	\]
	for each $n$.
\end{rem}

\begin{exmp}
	Take $\mu = 3/4$. Then the commuter $h_{3/4}$ is depicted below.
	\begin{figure}[htbp]
		\includegraphics[width=6.25cm]{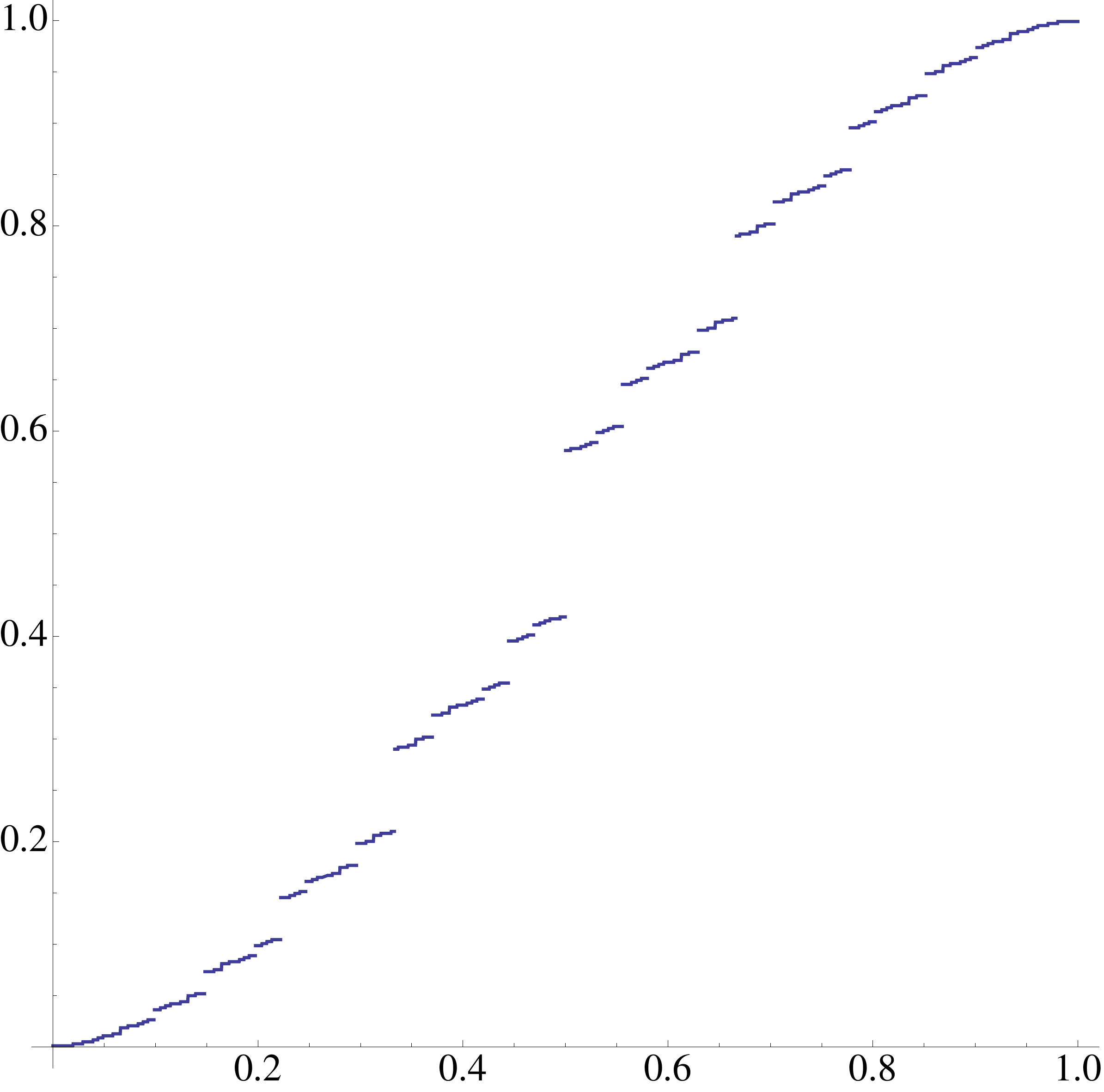}
		\caption{The commuter $h_{3/4}$. Notice that the function is highly discontinuous, and its range has the appearance of a Cantor set. However, it does
		appear to be increasing.}
	\end{figure}
\end{exmp}

Our ultimate goal is to prove that $\Allow(T_\mu) \subseteq \Allow(T)$ for all $\mu > 1/2$. To do this, we need to know that $h_\mu$ is order-preserving. Therefore, we 
now set about proving that $h_\mu$ is always strictly increasing for $\mu > 1/2$. We first develop some useful properties and then tackle the main proof.

\begin{lem}
	The function $h_\mu : [0,1] \to [0,1]$ is monotone increasing.
\end{lem}
\begin{proof}
	We begin by setting $f_0(x) =x$ and $f_{n+1}(x) = M_\mu f_n(x)$. We show by induction that each $f_n$ is strictly increasing, so $h_\mu = \lim f_n$ is, at the
	very least, monotone increasing.
	
	Certainly $f_0$ is strictly increasing. Suppose then that $f_{n-1}$ is strictly increasing. To show that $f_n$ is strictly increasing, we need to consider three cases.
	
	\begin{itemize}
		\item If $x < y \leq 1/2$, then we have
		\[
			f_n(x) = \tfrac{1}{2} f_{n-1}(2\mu x), \quad f_n(y) = \tfrac{1}{2} f_{n-1}(2\mu y).
		\]
		But $2\mu x < 2\mu y$, so $f_{n-1}(2\mu x) < f_{n-1}(2\mu y)$. Thus $f_n(x) < f_n(y)$.
	
		\item If $1/2 < x < y$, then we have
		\[
			f_n(x) = 1 - \tfrac{1}{2} f_{n-1}(2\mu(1-x))
		\]
		and
		\[
			f_n(y) = 1 - \tfrac{1}{2} f_{n-1}(2\mu(1-y)).
		\]
		Since $1-y < 1-x$, $f_{n-1}(2\mu(1-y)) < f_{n-1}(2\mu(1-x))$, so $f_n(x) < f_n(y)$.
	
		\item Suppose $x \leq 1/2 < y$. We have already established the fact that each $f_n$ maps $[0,1/2]$ to $[0,1/2]$ and $(1/2,1]$ to $[1/2,1]$. Thus we at
		least have $f_n(x) \leq 1/2 \leq f_n(y)$. Since $y > 1/2$, $2\mu(1-y) < \mu$, so $f_{n-1}(2\mu(1-y)) < f_{n-1}(\mu) \leq 1$. Therefore,
		\[
			f_n(y) > 1 - \tfrac{1}{2} f_{n-1}(\mu) \geq \tfrac{1}{2},
		\]
		so we indeed have $f_n(x) < f_n(y)$. 
	\end{itemize}
	
	Therefore, each $f_n$ is strictly increasing. Since $h_\mu = \lim_{n \to \infty} f_n$ is a uniform limit of increasing functions, it is increasing, and we are done.
\end{proof}

\begin{lem}
	If $\mu\in(0,1]$, we have $h_\mu(0) = 0$ and $h_\mu(1) = 1$.
\end{lem}
\begin{proof}
	We simply need to notice that
	\[
		h_\mu(0) = \tfrac{1}{2}h_\mu(2\mu \cdot 0) = \tfrac{1}{2} h_\mu(0),
	\]
	which forces $h_\mu(0)=0$. As a result,
	\[
		h_\mu(1) = 1 - \tfrac{1}{2}h_\mu(2\mu(1-1)) = 1 - \tfrac{1}{2}h_\mu(0) = 1. \qedhere
	\]
\end{proof}

\begin{lem}\label{lem:hmu not 1}
	If $1/2 < \mu < 1$, then $h_\mu(\mu) \neq 1$.
\end{lem}
\begin{proof}
	Suppose to the contrary that $h_\mu(\mu) = 1$. Then
	\[
		1 = h_\mu(\mu) = 1 - \tfrac{1}{2} h_\mu(2\mu(1-\mu)),
	\]
	so $h_\mu(2\mu(1-\mu)) = 0$. Also, we have
	\[
		h_\mu(\tfrac{1}{2}) = \tfrac{1}{2} h_\mu(2\mu \cdot \tfrac{1}{2}) = \tfrac{1}{2}h_\mu(\mu) = \tfrac{1}{2}
	\]
	and
	\[
		h_\mu(\tfrac{1}{4\mu}) = \tfrac{1}{2} h_\mu(2\mu \cdot \tfrac{1}{4\mu}) = \tfrac{1}{2} h_\mu(\tfrac{1}{2}) = \tfrac{1}{4}.
	\]
	Continuing, we see that
	\[
		h_\mu(\tfrac{1}{2^n \mu^{n-1}}) = \tfrac{1}{2^n}
	\]
	in general. Choose $n$ sufficiently large to ensure that
	\[
		0 < \frac{1}{2^n \mu^{n-1}} < 2\mu(1-\mu).
	\]
	Then
	\[
		0 = h_\mu(0) < h_\mu(\tfrac{1}{2^n \mu^{n-1}}) > h_\mu(2\mu(1-\mu)) = 0,
	\]
	contradicting the fact that $h_\mu$ is monotone increasing. Therefore, we must have $h_\mu(\mu) \neq 1$.
\end{proof}

\begin{lem}
	If $1/2 < \mu < 1$, then $h_\mu$ has a jump discontinuity at $x=1/2$.
\end{lem}
\begin{proof}
	Put $a = h_\mu(\tfrac{1}{2})$. Since $h_\mu$ is increasing,
	\[
		\lim_{x \to \frac{1}{2}^-} h_\mu(x) \leq a.
	\]
	The functional equation then implies that
	\[
		\lim_{x \to \frac{1}{2}^+} h_\mu(x) = 1-a.
	\]
	Thus $h_\mu$ is continuous at $x=1/2$ if and only if $a = 1/2$. But if $h_\mu(\tfrac{1}{2}) = \tfrac{1}{2}$, then
	\[
		h_\mu(\mu) = h_\mu(2\mu \cdot \tfrac{1}{2}) = 2 h_\mu(\tfrac{1}{2}) = 1,
	\]
	which is impossible by Lemma \ref{lem:hmu not 1}. Thus $h_\mu$ is discontinuous at $x=1/2$.
\end{proof}

The next lemma shows that $h_\mu$ has jump discontinuities corresponding to the peaks of all the iterates of $T_\mu$. More precisely, we claim that $h_\mu$
is discontinuous at any point where $T_\mu^n$ attains a local maximum for some $n$. Since $T_\mu^n$ is piecewise monotone (indeed, piecewise linear), the local 
extrema occur precisely at the points where $T_\mu^n$ is not differentiable. These are exactly the points $x_0$ for which $T_\mu^{n-1}(x_0)=1/2$ or $T_\mu^{n-1}$ is 
not differentiable at $x_0$. Inductively, these points are just the preimages of $1/2$ under the maps $T_\mu^0, T_\mu, T_\mu^2, \ldots, T_\mu^{n-1}$.

\begin{lem} 
\label{lem: jump disc}
	Suppose $1/2 < \mu < 1$, and let $x_0 \in [0,1]$. If there exists $n \geq 0$ such that $T_\mu^n(x_0)=1/2$, then $h_\mu$ is discontinuous at $x_0$.
\end{lem}
\begin{proof}
	Let $n$ be an integer for which $T_\mu^n(x_0) = 1/2$. Note first that
	\[
		T^n(h_\mu(x_0)) = h_\mu(T_\mu^n(x_0)) = h_\mu(\tfrac{1}{2}).
	\]
	Now observe that
	\[
		T^n \Bigl( \lim_{x \to x_0^-} h_\mu(x) \Bigr) = \lim_{x \to x_0^-} T^n(h_\mu(x)) = \lim_{x \to x_0^-} h_\mu(T_\mu^n(x)) = \lim_{x \to \frac{1}{2}^-} h_\mu(x),
	\]
	and similarly,
	\[
		T^n \Bigl( \lim_{x \to x_0^+} h_\mu(x) \Bigr) = \lim_{x \to x_0^+} T^n(h_\mu(x)) = \lim_{x \to x_0^+} h_\mu(T_\mu^n(x)) = \lim_{x \to \frac{1}{2}^+} h_\mu(x).
	\]
	Since $h_\mu$ has a jump discontinuity at $1/2$, it follows that
	\[
		\lim_{x \to x_0^-} h_\mu(x) \neq  \lim_{x \to x_0^+} h_\mu(x).
	\]
	Thus $h_\mu$ has a jump discontinuity at $x_0$.
\end{proof}

\begin{lem}\label{lem: x0 btwn x and y}
	If $x, y \in [0,1]$ with $x < y$, then there exists $x < x_0 < y$ such that $T_\mu^n(x_0) = 1/2$ for some $n \geq 0$.
\end{lem}
\begin{proof}
	Put $\delta = \abs{x-y}$. Assume first that $x < y \leq 1/2$. Since $T_\mu$ is continuous and strictly increasing on $[0, 1/2]$, $T_\mu((x,y)) = (2\mu x, 2\mu y)$. 
	Similarly, if $1/2 < x < y$, then $T_\mu((x,y)) = (2\mu(1-y), 2\mu(1-x))$. In either case, $T_\mu$ stretches $(x,y)$ by a factor of $2\mu$ (since $\mu > 1/2$). If 
	$T_\mu((x,y))$ is contained entirely within either $[0,1/2]$ or $[1/2,1]$, apply $T_\mu$ again, which stretches the interval by another factor of $2\mu$. Repeat 
	until $1/2 \in T_\mu^n((x,y))$. This process is guaranteed to terminate before $n = \lceil -\log(2\delta)/\log(2\mu) \rceil$. Indeed, if $1/2 \not\in T_\mu^k((x,y))$ for 
	$k < n = \lceil -\log(2\delta)/\log(2\mu) \rceil$, then $T_\mu^n((x,y))$ is guaranteed to have length 
	\[
		(2\mu)^n\delta > (2\mu)^{-\log(2\delta)/\log(2\mu)} \delta = e^{-\log(2\delta)} \delta = \tfrac{1}{2},
	\]
	which forces $1/2 \in T_\mu^n((x,y))$. Thus there exists $x_0 \in (x,y)$ such that $T_\mu^n(x_0)=1/2$ for some $n \geq 1$.
\end{proof}

\begin{cor}\label{cor: jump disc}
	Given two points $x, y \in [0,1]$ with $x < y$, there exists $x_0 \in (x,y)$ such that $h_\mu$ is discontinuous at $x_0$.
\end{cor}
\begin{proof}
	We have just shown in Lemma \ref{lem: x0 btwn x and y} that between any two points $x, y \in [0,1]$, we can find a point $x_0 \in (x,y)$ such that $T_\mu^n(x_0)=1/2$ for some $n$. But we have also
	shown in Lemma \ref{lem: jump disc} that $h_\mu$ has a jump discontinuity at any such point.
\end{proof}

\begin{thm}\label{thmInc}
	The function $h_\mu$ is strictly increasing on $[0,1]$.
\end{thm}
\begin{proof}
	Let $x, y \in [0,1]$ with $x < y$. By Corollary \ref{cor: jump disc}, there is a point $x_0$ between $x$ and $y$ at which $h_\mu$ has a jump discontinuity. Since 
	$h_\mu$ is increasing, we have
	\[
		h_\mu(x) \leq \lim_{t \to x_0^-} h_\mu(t) < \lim_{t \to x_0^+} h_\mu(t) \leq h_\mu(y),
	\]
	so $h_\mu$ is indeed strictly increasing.
\end{proof}

We close this section with a useful result about the family $\{h_\mu\}$ of commuters for $1/2 < \mu < 1$. One would expect that the functions $h_\mu$ 
should approach the identity function $h(x)=x$ as $\mu \to 1$, at least pointwise. In fact, we prove that $h_\mu \to h$ \emph{uniformly} as $\mu \to 1$. 

Recall that we established the existence of $h_\mu$ by defining it to be the unique fixed point of the contraction $M_\mu : \F \to \F$.
Not only is each $M_\mu$ contractive, but the one-parameter family $\{M_\mu\}_{\frac{1}{2} < \mu < 1}$ is \emph{uniformly contractive} in the sense that
\[
	\norm{M_\mu f - M_\mu g}_\infty \leq \alpha \norm{f-g}_\infty
\]
for all $f, g \in \F$, where $\alpha$ is a constant that is independent of $\mu$. In particular, we can take $\alpha = 1/2$. Also, notice that for $\mu=1$ the contraction 
$M := M_1$ takes the form
\[
	M f(x) = \left\{ \begin{array}{cc}
		\frac{1}{2} f(2 x) & \text{ if } 0 \leq x \leq 1/2 \\
		1-\frac{1}{2}f(2 (1-x)) & \text{ if } 1/2 < x \leq 1
	\end{array}
	\right.
\]
and the identity function $h(x) = x$ is the unique fixed point of $M$. With these facts in hand, we are now in a position to invoke the Uniform Contraction Principle 
of \cite{stuart-humph} to see that $h_\mu \to h$ uniformly.

\begin{thm}
\label{thm:unifcont}
	As $\mu \to 1$, the one-parameter family $\{h_\mu\}$ converges uniformly to the identity function $h : [0,1] \to [0,1]$.
\end{thm}
\begin{proof}
	We have already seen that the family $\{M_\mu\}_{\tfrac{1}{2} < \mu \leq 1}$ is uniformly contractive with contraction constant $\alpha = 1/2$. Now we claim that for 
	each $\mu$,
	\[
		\norm{M_\mu h - h} \leq \frac{1-\mu}{2}. 
	\]
	If $x \in [0, \frac{1}{2}]$, then
	\[
		\abs{M_\mu h(x) - h(x)} = \abs{\tfrac{1}{2} h(2\mu x) - h(x)} = \abs{\tfrac{1}{2} \cdot 2 \mu x - x} = \abs{(\mu - 1)x} \leq \frac{1-\mu}{2}.
	\]
	Likewise, if $x \in (\frac{1}{2},1]$, then
	\begin{align*}
		\abs{M_\mu h(x) - h(x)} &= \abs{1-\tfrac{1}{2} h(2\mu(1- x)) - h(x)} \\
			&= \abs{1-\mu+\mu x - x} \\
			&= \abs{(1-\mu)(1-x)} \\
			&\leq \frac{1-\mu}{2}.
	\end{align*}
	The Uniform Contraction Principle \cite[Theorem C.5]{stuart-humph} now guarantees that 
	\[
		\norm{h_\mu - h}_\infty \leq \frac{1-\mu}{2} \cdot \frac{1}{1-\tfrac{1}{2}} = 1-\mu.
	\]
	From this it is clear that $h_\mu \to h$ uniformly as $\mu \to 1$.
\end{proof}


\section{The Range of $h_\mu$}\label{secRange}
It is particularly interesting to study the range of the map $h_\mu$ since we can see from the proof of Theorem \ref{thm:semiconjugate_patterns} that the allowed permutations realized by the map $T_\mu$ are exactly $$\{ \pi \mid \pi = \Pat(x,T,n\} \text{ for } x \in \text{Range}(h_\mu)\} \subseteq \Allow_n(T).$$ 

Based on the pictures above, it appears that the range of $h_\mu$ is a Cantor-like set. In particular, it looks as though the gap at $x=1/2$ is replicated at smaller and 
smaller scales throughout the range of $h_\mu$. Indeed, we have already seen that this jump discontinuity is replicated at precisely the points where the peaks of the 
iterates of $T_\mu$ occur. We aim to show here that the gaps in the range consist of a union of intervals centered at dyadic rationals, each with radius proportional to 
that of the gap at $x=1/2$.

We begin by observing that the range of $h_\mu$ must exclude any point $y$ for which $T(y) > h_\mu(\mu)$. This is due to the commutation relationship
\[
	h_\mu \circ T_\mu = T \circ h_\mu.
\]
Since the maximum of $T_\mu$ is $\mu$, the possible values of the left side are at most $h_\mu(\mu)$. The standard tent map $T$ takes values greater than 
$h_\mu(\mu)$ whenever $x$ is between $h_\mu(\mu)/2$ and $1-h_\mu(\mu)/2$, so the interval
\[
	\left( \frac{h_\mu(\mu)}{2}, 1 - \frac{h_\mu(\mu)}{2} \right)
\]
is omitted from the range of $h_\mu$. We also have
\[
	h_\mu \circ T_\mu^2 = T^2 \circ h_\mu,
\]
so $h_\mu$ can never take values in the set $(T^2)^{-1}((h_\mu(\mu),1])$. Thus
\[
	\left( \frac{h_\mu(\mu)}{4}, \frac{1}{2} - \frac{h_\mu(\mu)}{4} \right) \cup \left( \frac{1}{2} + \frac{h_\mu(\mu)}{4}, 1-\frac{h_\mu(\mu)}{4} \right)
\]
is excluded from the range of $h_\mu$. In general, $h_\mu$ cannot take values that would make $T^n$ greater than $h_\mu(\mu)$. We prove below that this occurs 
on the set
\[
	\bigcup_{i=1}^{2^{n-1}} \left( \frac{2i-1}{2^n} - \frac{1-h_\mu(\mu)}{2^n}, \frac{2i-1}{2^n} + \frac{1-h_\mu(\mu)}{2^n} \right).
\]

\begin{prop}
\label{prop:range}
	The set
	\[
		\bigcup_{n=1}^\infty \bigcup_{i=1}^{2^{n-1}} \left( \frac{2i-1}{2^n} - \frac{1-h_\mu(\mu)}{2^n}, \frac{2i-1}{2^n} + \frac{1-h_\mu(\mu)}{2^n} \right)
	\]
	does not belong to the range of $h_\mu$.
\end{prop}
\begin{proof}
	First recall that for each $n \geq 1$, the peaks of $T^n$ (i.e., the points where $T^n(x) = 1$) occur at the dyadic points $x = \frac{2i-1}{2^n}$ for $1 \leq i \leq 2^{n-1}$.
	On the interval $[0, \frac{1}{2^n}]$ we have
	\[
		T^n(x) = 2^n x,
	\]
	so $T^n(x) > h_\mu(\mu)$ when $\frac{h_\mu(\mu)}{2^n} < x \leq \frac{1}{2^n}$. Similarly, on the interval $[\frac{1}{2^n}, \frac{1}{2^{n-1}}]$
	\[
		T^n(x) = 2^n(\tfrac{1}{2^{n-1}}-x),
	\]
	so $T^n(x) > h_\mu(\mu)$ when $\frac{1}{2^n} \leq x < \frac{1}{2^{n-1}} - \frac{h_\mu(\mu)}{2^n}$. Therefore, $T^n(x) > h_\mu(\mu)$ for all $x$ in the interval
	\[
		\left( \frac{h_\mu(\mu)}{2^n}, \frac{1}{2^{n-1}} - \frac{h_\mu(\mu)}{2^n} \right).
	\]
	This interval is symmetric about $\frac{1}{2^n}$, which we can make more evident by rewriting it as
	\[
		\left( \frac{1}{2^n} - \frac{1 - h_\mu(\mu)}{2^n}, \frac{1}{2^{n}} + \frac{1-h_\mu(\mu)}{2^n} \right).
	\]
	We can obtain the intervals around the other peaks by simply translating. That is, we have $T^n(x) > h_\mu(\mu)$ for all $x$ in the intervals
	\[
		\left( \frac{2i-1}{2^n} - \frac{1 - h_\mu(\mu)}{2^n}, \frac{2i-1}{2^{n}} + \frac{1-h_\mu(\mu)}{2^n} \right), \quad 1 \leq i \leq 2^{n-1}.
	\]
	Thus none of these intervals can belong to the range of $h_\mu$. Taking the union over $1 \leq i \leq 2^{n-1}$ and over all $n \geq 1$ yields the desired result.
\end{proof}

\section{Allowed and Forbidden Patterns}\label{secAllowed}
Here, we study the relationship between the allowed and forbidden patterns of $T_\mu$ and $T$, starting with the following theorem which tells us that any pattern realized by $T_\mu$ for $1/2<\mu\leq 1$ must also be realized by $T$.

\begin{thm}\label{thm:main}
	Suppose $\pi \in \allow(T_\mu)$, where $1/2 < \mu \leq 1$. Then $\pi \in \allow(T)$.
\end{thm}
\begin{proof}
	Since $h_\mu$ is increasing by Theorem \ref{thmInc}, we could take $f = T$, $g = T_\mu$ and $h = h_\mu$ in Theorem \ref{thm:semiconjugate_patterns}. The result follows. 
	\end{proof}

Now we set out to investigate the length of the shortest pattern allowed for the full tent map $T$ but forbidden for $T_\mu$. This requires us to 
more closely analyze the behavior of $T$ and its iterates near $x=1/2$. Consequently, we show that the pattern of length $n$ realized at points 
near $1/2$ always has a very particular form. Moreover, this pattern can only occur near $1/2$.

\begin{prop}
\label{prop:Tallowed}
	Fix $n \geq 3$. Then for all $x \in \bigl( \frac{2^{n-2}}{2^{n-1}+1}, \frac{2^{n-2}}{2^{n-1}-1}\bigr) \backslash \{\frac{1}{2} \}$, 
	\[
		\pat(x,T,n) = (n-1)n 1 2 3 \cdots (n-2).
	\]
\end{prop}
\begin{proof}
	First notice that if $m \geq 2$, then $T^m(\tfrac{1}{2}) = 0$. Also, $T^m$ is a piecewise linear function with slope $\pm 2^m$. Thus the monotone segment of $T^m$
	to the left of $x=1/2$ is
	\begin{equation}
	\label{eq:piecewise1}
		y = -2^m \left( x-\tfrac{1}{2} \right), \quad \tfrac{2^{m-1}-1}{2^m} \leq x \leq \tfrac{1}{2},
	\end{equation}
	while the segment to the right is
	\begin{equation}
	\label{eq:piecewise2}
		y = 2^m \left( x-\tfrac{1}{2} \right), \quad \tfrac{1}{2} \leq  x \leq \tfrac{2^{m-1}+1}{2^m}.
	\end{equation}
	It follows that for all $x \in \bigl( \frac{2^{n-2}-1}{2^{n-1}}, \frac{2^{n-2}+1}{2^{n-1}} \bigr) \backslash \{\frac{1}{2} \}$,
	\[
		T^2(x) < T^3(x) < \cdots < T^{n-1}(x).
	\]
	To finish the proof, it suffices to find a (possibly smaller) interval on which
	\[
		T^{n-1}(x) < x < T(x).
	\]
	Note first that $x < T(x)$ for all $x \in (0, \frac{2}{3})$. Now, simply set \eqref{eq:piecewise1} and \eqref{eq:piecewise2} equal to $x$ (taking $m=n-1$) and solve. This 
	yields
	\[
		x = \frac{2^{n-2}}{2^{n-1}+1}, \; \frac{2^{n-2}}{2^{n-1}-1}.
	\]
	Notice that $\frac{2^{n-2}}{2^{n-1}-1} < \frac{2}{3}$ for all $n \geq 3$. Thus for all $x \in \bigl( \frac{2^{n-2}}{2^{n-1}+1}, \frac{2^{n-2}}{2^{n-1}-1}\bigr) \backslash 
	\{\frac{1}{2} \}$, we have
	\[
		T^2(x) < T^3(x) < \cdots < T^{n-1}(x) < x < T(x),	
	\]
	so $\pat(x, T, n) = (n-1)n123\cdots(n-2)$.
\end{proof}

\begin{prop}
	The pattern $(n-1)n123 \cdots (n-2)$ is realized nowhere else.
\end{prop}
\begin{proof}
	We proceed by induction on $n$. Notice first that the pattern $\sigma_3 = 231$ is realized only on the interval $(\tfrac{2}{5}, \tfrac{2}{3})$. Assume
	$\sigma_{n-1} = (n-2)(n-1)123 \cdots (n-3)$ is realized only on the interval $\bigl( \frac{2^{n-3}}{2^{n-2}+1}, \frac{2^{n-3}}{2^{n-2}-1}\bigr)$. Then $\sigma_n$
	cannot occur outside this interval, since the first $n-1$ terms of $\sigma_n$ are in the same relative order as $\sigma_{n-1}$. As stated in the proof of Proposition 
	\ref{prop:Tallowed}, $T^{n-1}(x) = x$ at the points $\frac{2^{n-2}}{2^{n-1}+1}$ and $\frac{2^{n-2}}{2^{n-1}-1}$, and $T^{n-1}(x) < x$ when $x$ lies between these points. By
	\eqref{eq:piecewise1}, $T^{n-1}$ is linear on the interval $\bigl(\frac{2^{n-2}+1}{2^{n-1}}, \frac{1}{2} \bigr)$, therefore $x < T^{n-1}(x)$ when $x \in \bigl( \frac{2^{n-2}+1}{2^{n-1}},
	\frac{2^{n-2}}{2^{n-1}+1} \bigr)$. Likewise, it follows from \eqref{eq:piecewise2} that $x < T^{n-1}(x)$ 
	when $x \in \bigl( \frac{2^{n-2}}{2^{n-1}-1}, \frac{2^{n-2}+1}{2^{n-1}} \bigr)$. Thus $\sigma_n$ cannot be realized on $\bigl( \frac{2^{n-2}+1}{2^{n-1}}, \frac{2^{n-2}}{2^{n-1}+1} \bigr)$ 
	or $\bigl( \frac{2^{n-2}}{2^{n-1}-1}, \frac{2^{n-2}+1}{2^{n-1}} \bigr)$. It is straightforward to check that
	\[
		\frac{2^{n-2}+1}{2^{n-1}} < \frac{2^{n-3}}{2^{n-2}+1} \quad \text{and} \quad \frac{2^{n-3}}{2^{n-2}-1} < \frac{2^{n-2}+1}{2^{n-1}},
	\]
	so it follows that the only points of $\bigl( \frac{2^{n-3}}{2^{n-2}+1}, \frac{2^{n-3}}{2^{n-2}-1}\bigr)$ satisfying $x < T^{n-1}(x)$ lie in the smaller interval 
	$\bigl( \frac{2^{n-2}}{2^{n-1}+1}, \frac{2^{n-2}}{2^{n-1}-1}\bigr)$. Therefore, $\sigma_n$ is realized only on this interval.
\end{proof}

Thanks to Proposition \ref{prop:range}, we know many values that are omitted from the range of $h_\mu$ when $1/2 < \mu < 1$. We can 
use this information to determine conditions for when $T_\mu$ avoids the pattern $\sigma_n = (n-1)n123 \cdots (n-2)$ from the previous two propositions.

\begin{cor}
\label{cor:forbidden}
	If
	\[
		h_\mu(\mu) < 2 \left( 1 - \frac{2^{n-2}}{2^{n-1}-1} \right),
	\]
	then $T_\mu$ avoids the pattern $\sigma_n = (n-1)n123\cdots(n-2)$.
\end{cor}
\begin{proof}
	We already know that the range of $h$ omits the interval 
	\[
		\left( \frac{h_\mu(\mu)}{2}, 1- \frac{h_\mu(\mu)}{2} \right)
	\]
	and that $\sigma_n$ is only realized on the interval $\bigl( \frac{2^{n-2}}{2^{n-1}+1}, \frac{2^{n-2}}{2^{n-1}-1}\bigr) \backslash \{\frac{1}{2} \}$. In
	light of this, it suffices to show that
	\[
		\frac{h_\mu(\mu)}{2} < \frac{2^{n-2}}{2^{n-1}+1} \quad \text{and} \quad 1 - \frac{h_\mu(\mu)}{2} > \frac{2^{n-2}}{2^{n-1}-1}.
	\]
	The latter inequality is immediate from our hypothesis. We can get the first inequality from the second by simply reflecting over the line 
	$x=1/2$:
	\[
		\frac{h_\mu(\mu)}{2} < 1-\frac{2^{n-2}}{2^{n-1}-1} = \frac{2^{n-1}-2^{n-2}-1}{2^{n-1}-1}.
	\]
	Since
	\begin{align*}
		(2^{n-1}-2^{n-2}-1)(2^{n-1}+1) &= 2^{2n-2} - 2^{2n-3} -2^{n-2} - 1 \\
			&= 2^{2n-3} - 2^{n-2} - 1 \\
			& < 2^{2n-3} - 2^{n-2} \\
			& = 2^{n-2}(2^{n-1}-1),
	\end{align*}
	it follows that 
	\[
		\frac{h_\mu(\mu)}{2} < \frac{2^{n-2}}{2^{n-1}+1},
	\]
	and we are done.
\end{proof}

Given the inherent mystery surrounding the functions $h_\mu$, it would be nice if we could somehow obtain a bound involving $\mu$ itself that
would guarantee $T_\mu$ avoids $\sigma_n$. To do so, we first need to relate $h_\mu(\mu)$ to $\mu$. This involves a more careful implementation of
the estimates in the proof of Theorem \ref{thm:unifcont}.

\begin{prop}
\label{prop:bound}
	For all $\mu \in (1/2,1]$, $\abs{h_\mu(\mu) - \mu} \leq \frac{1}{2}(1-\mu) + (1-\mu)^2$.
\end{prop}
\begin{proof}
	Notice first that
	\[
		\abs{h_\mu(x) - x} \leq \abs{h_\mu(x) - M_\mu h(x)} + \abs{M_\mu h(x) - x}
	\]
	for all $x \in [0,1]$. Since each $M_\mu$ is a contraction with contraction constant $1/2$, we have
	\[
		\abs{h_\mu(x) - M_\mu h(x)} = \abs{M_\mu h_\mu(x) - M_\mu h(x)} \leq \tfrac{1}{2} \norm{h_\mu - h}_\infty \leq \tfrac{1}{2}(1-\mu).
	\]
	Moreover, if $x \in (1/2,1]$, we have
	\[
		\abs{M_\mu h(x) - x} = (1-\mu)(1-x)
	\]
	from the proof of Theorem \ref{thm:unifcont}. It follows then that
	\[
		\abs{h_\mu(\mu) - \mu} \leq \tfrac{1}{2}(1-\mu) + (1-\mu)^2.  \qedhere
	\]
\end{proof}

We can now couple this estimate with Corollary \ref{cor:forbidden} to obtain a bound in terms of $\mu$ that guarantees the avoidance of certain patterns by $T_\mu$.

\begin{thm}
\label{thm:mubound}
	Fix $n >5$. If
	\begin{equation}
	\label{eq:inequality}
		\mu < \frac{3}{4} + \frac{1}{4} \sqrt{9-\frac{2^{n+2}+8}{2^{n-1}-1}},
	\end{equation}
	then $T_\mu$ avoids the pattern $\sigma_n = (n-1)n123\cdots(n-2)$.
\end{thm}
\begin{proof}
	We know from Corollary \ref{cor:forbidden} that $\sigma_n$ is avoided by $T_\mu$ if 
	\[
		h_\mu(\mu) < 2 \left(1-\frac{2^{n-2}}{2^{n-1}-1} \right) = 2 - \frac{2^{n-1}}{2^{n-1}-1}.
	\]
	But the previous proposition shows that $h_\mu(\mu) \leq \frac{1}{2}(1-\mu) + (1-\mu)^2 + \mu$, so
	\begin{equation}
	\label{eq:guarantee}
		\tfrac{1}{2}(1-\mu) + (1-\mu)^2 + \mu < 2 - \frac{2^{n-1}}{2^{n-1}-1}
	\end{equation}
	would guarantee that $T_\mu$ avoids $\sigma_n$. This inequality is equivalent to
	\[
		\mu^2 - \frac{3}{2}\mu + \frac{2^{n-1}+1}{2^n-2} < 0.
	\]
	The roots of this quadratic are precisely
	\[
		\mu = \frac{3}{4} \pm \frac{1}{4} \sqrt{9-\frac{2^{n+2}+8}{2^{n-1}-1}},
	\]
	which are real provided $n > 5$. Thus \eqref{eq:guarantee} is satisfied whenever
	\[
		\frac{3}{4} - \frac{1}{4} \sqrt{9-\frac{2^{n+2}+8}{2^{n-1}-1}} < \mu < \frac{3}{4} + \frac{1}{4} \sqrt{9-\frac{2^{n+2}+8}{2^{n-1}-1}}.
	\]
	The first term is always less than $1/2$, so we are simply left with \eqref{eq:inequality}. The result then follows.
\end{proof}

Notice that this theorem implies that the inclusion $\Allow(T_\mu) \subseteq \Allow(T_1)$ in Theorem \ref{thm:main} is strict when $\mu<1$. Indeed for any $\mu<1$, there is a sufficiently large $n$ so that Theorem \ref{thm:mubound} implies that $T_\mu$ avoids $\sigma_n$, while such patterns belong to $\Allow(T)$ for all $n$.
As discussed in the next section, the patterns $\sigma_n$ are of particular interest, as we conjecture that the smallest pattern allowed by $T$ and avoided by $T_\mu$ is of the form $\sigma_n$ for some $n$. 

For small values of $n$, we can compute
\[
	\sup \bigl\{ \mu : \sigma_n \not \in \Allow(T_\mu) \bigr\}
\]
exactly. We present these values for $4 \leq n \leq 12$ in Table \ref{tab:table}, together with the upper bounds computed using Theorem \ref{thm:mubound}. (We omit the case
$n=3$, since $\sigma_3 = 231$ is an allowed pattern of $T_\mu$ for $1/2 < \mu \leq 1$.)


\begin{table}
\centering
\setlength{\tabcolsep}{2em}
\begin{tabular}{|c|c|c|}
	\hline
	$n$ & $\mu_{n,a}$ & $\mu_{n,e}$ \\
	\hline
	4 & 0.809017 & --- \\
	\hline
	5 & 0.919643 & --- \\
	\hline
	6 & 0.963781 & 0.923902 \\
	\hline
	7 & 0.982974 & 0.965933 \\
	\hline
	8 & 0.991791 & 0.983722 \\
	\hline
	9 & 0.995982 & 0.992030 \\
	\hline
	10 & 0.998016 & 0.996055 \\
	\hline
	11 & 0.999015 & 0.998037 \\
	\hline
	12 & 0.999509 & 0.999021 \\
	\hline
\end{tabular}
\vspace{3mm}
\caption{This table depicts the true and estimated upper bounds on $\mu$ (to six decimal places) that guarantee $T_\mu$ avoids $\sigma_n$ for some specific values of $n$. Here 
$\mu_{n,a}$ is the true upper bound (i.e., $T_\mu$ avoids $\sigma_n$ if and only if $\mu < \mu_{n, a}$) while $\mu_{n,e}$ is the upper bound afforded by Theorem \ref{thm:mubound}.}
\label{tab:table}
\end{table}

\section{Conjectures}\label{secOpen}
We now state some conjectures related to this work. Given $\mu \in (1/2, 1)$, we define a pattern $\pi$ to be 
\emph{$\mu$-forbidden} if $\pi \in \Allow(T)$ but $\pi \not\in \Allow(T_\mu)$.

Our first conjecture is that the shortest pattern avoided by $T_\mu$, but allowed by $T$, can always be taken to be of the form $\sigma_n = (n-1)n123\cdots(n-2)$. In other 
words, there may be other patterns of the same length that are $\mu$-forbidden, but none shorter than the shortest $\sigma_n$ that is $\mu$-forbidden.

\begin{conj}
	For any $1/2 < \mu < 1$, the shortest $\mu$-forbidden pattern is of the form
	\[
		\sigma_n = (n-1)n123\cdots(n-2).
	\]
	That is, if $n$ is the length of the shortest $\mu$-forbidden pattern, then $T_\mu$ avoids $\sigma_n$.
\end{conj}

In addition to numerical evidence, this conjecture is supported by the observation that the behavior of the iterates of $T_\mu$ differs the most from that of the iterates of $T$ near 
$x=1/2$ (as in Figure \ref{fig:tent}). Therefore, we expect the shortest $\mu$-forbidden pattern to have the form $\Pat(x, T, n)$ for some $n \in \mathbb{N}$ and $x$ in a sufficiently small 
neighborhood of $1/2$. But Proposition \ref{prop:Tallowed} shows that $\Pat(x, T, n) = \sigma_n$ when $x$ is close to $1/2$.

Our second conjecture involves the relationship between the allowed patterns of two tent maps $T_\mu$ and $T_\nu$, where $\mu < \nu$. We 
already know that if $\nu=1$, then
\[
	\Allow(T_\mu) \subseteq \Allow(T_\nu).
\]
We would expect something like this to be true in general, though the iterative process for building commuters falls apart here. However, a closer
analysis of the commuters $h_\mu$ and $h_\nu$, together with Proposition \ref{prop:range}, should yield a positive result.

\begin{conj}
	If $1/2 < \mu < \nu \leq 1$, then $h_\mu(\mu) < h_\nu(\nu)$. Consequently, the range of $h_\mu$ is contained in the range of $h_\nu$,
	and we have
	\[
		\Allow(T_\mu) \subsetneq \Allow(T_\nu).
	\]
\end{conj}

To obtain a positive resolution to this conjecture, it is necessary for one to show that $h_\mu(\mu)$ is increasing with $\mu$. Numerical evidence suggests that
this is the case (see Figure \ref{fig:figure3}).

Finally, one would hope for a tighter bound than the one obtained in Proposition \ref{prop:bound}. Numerical evidence indicates that there is a better bound. However, we 
are unable to prove it at this time.

\begin{conj}
	The bound in Proposition \ref{prop:bound} can be improved. In particular, for all $\mu \in (\frac{1}{2}, 1]$ we have
	\[
		h_\mu(\mu) \leq \mu^2 + \tfrac{5}{4}(1-\mu).
	\]
\end{conj}

\begin{figure}[h]
	\includegraphics[width=6cm]{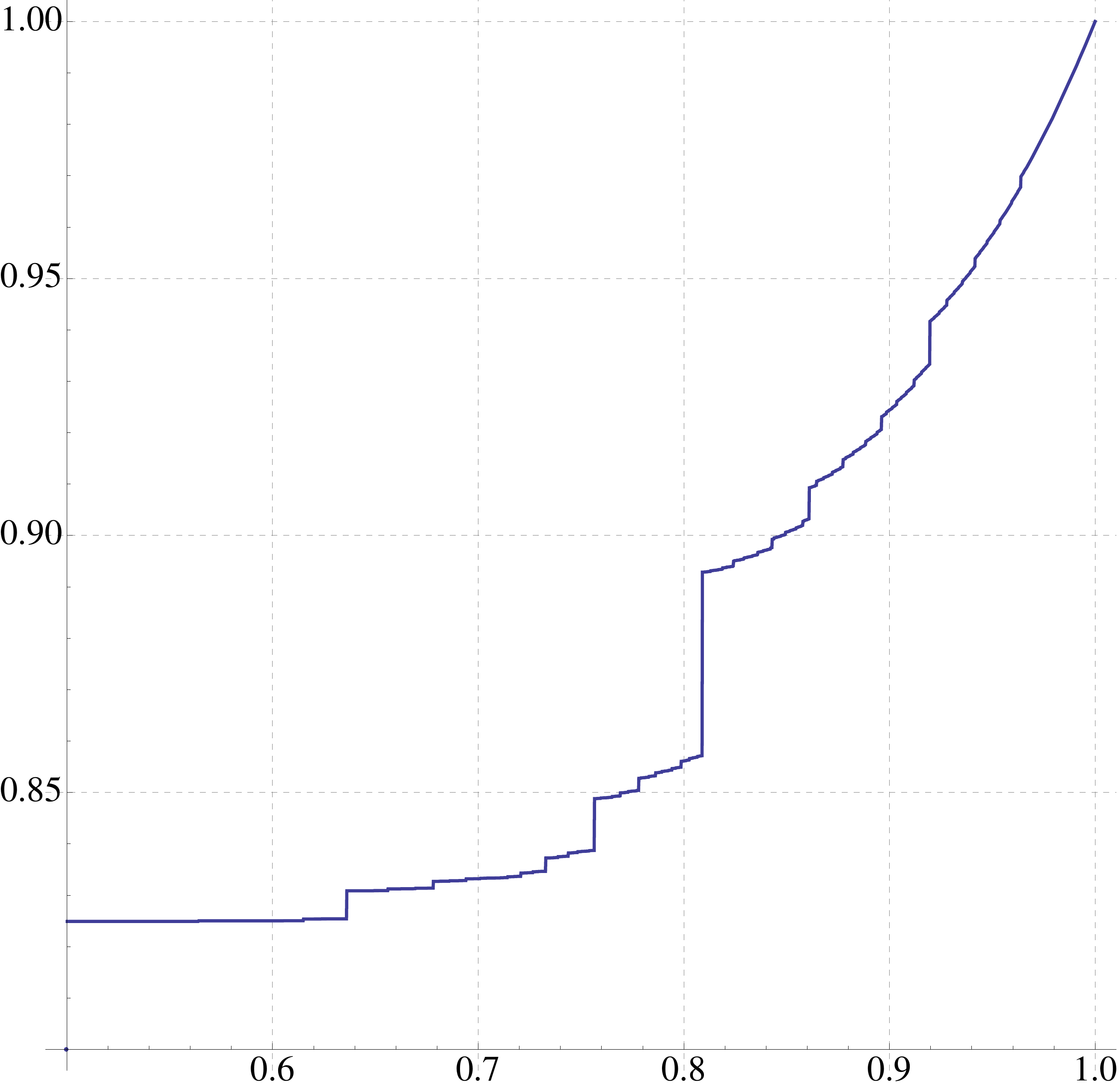}
	\hspace{2mm}
	\includegraphics[width=6cm]{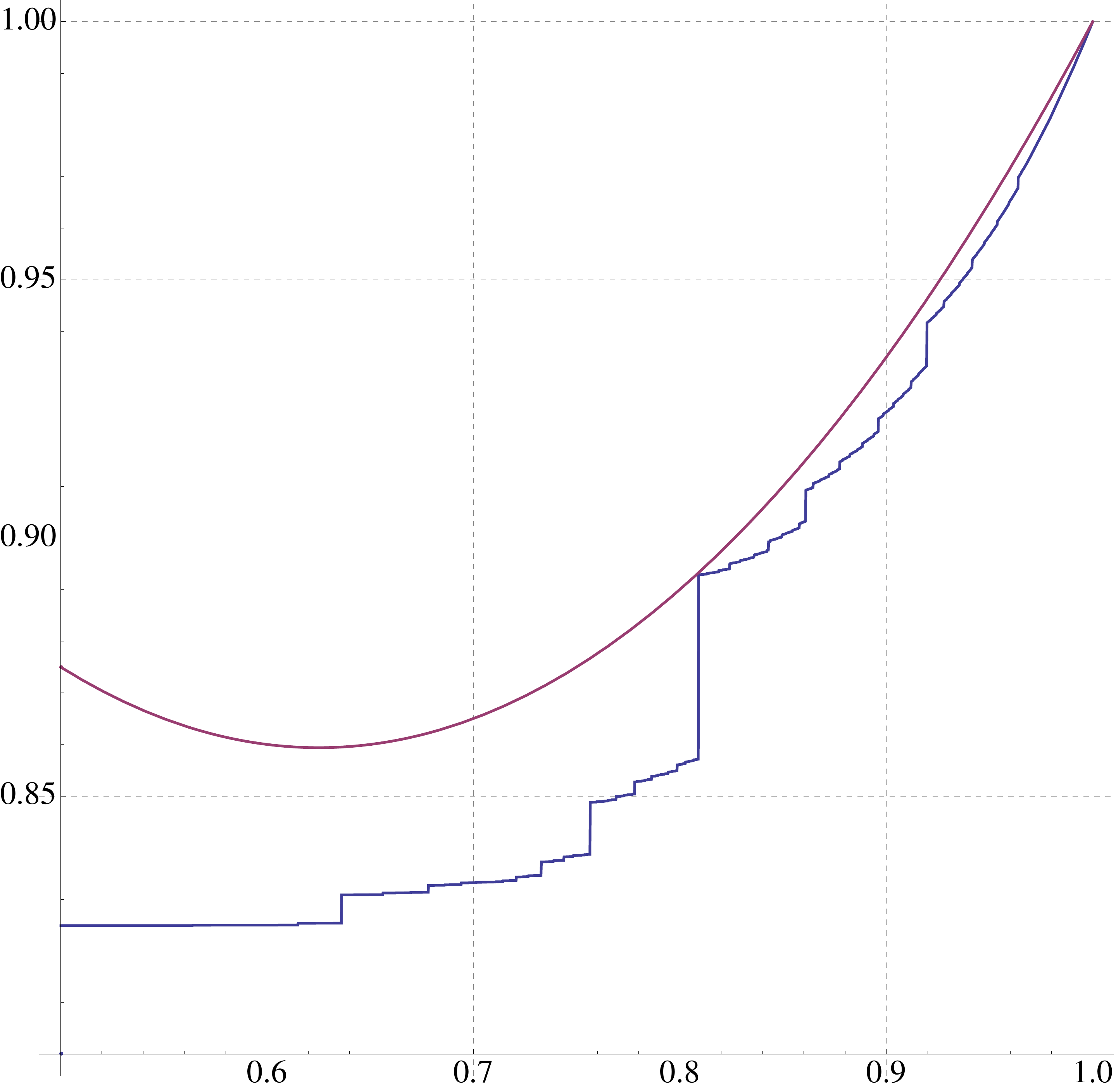}
	\caption{(Left) A plot of $h_\mu(\mu)$ versus $\mu$, for $1/2 < \mu \leq 1$, which supports Conjecture 2. (Right) A plot of $h_\mu(\mu)$ together with 
	$\mu^2 + \tfrac{5}{4}(1-\mu)$ for $1/2 < \mu \leq 1$, which appears to corroborate Conjecture 3.}
	\label{fig:figure3}
\end{figure}

\subsection*{Acknowledgements} The authors would like to thank the anonymous referees for helpful suggestions that improved the final version of the paper.

\bibliographystyle{amsplain}
\bibliography{ComboBib.bib}

\providecommand{\bysame}{\leavevmode\hbox to3em{\hrulefill}\thinspace}
\providecommand{\MR}{\relax\ifhmode\unskip\space\fi MR }
\providecommand{\MRhref}[2]{%
  \href{http://www.ams.org/mathscinet-getitem?mr=#1}{#2}
}
\providecommand{\href}[2]{#2}
\begin{thebibliography}{10}

\bibitem{Amigosigned}
J.M. Amig\'o, \emph{The ordinal structure of the signed shift transformations},
  Internat. J. Bifur. Chaos Appl. Sci. Engrg. (2009), no.~19, 3311--3327.

\bibitem{AEK}
J.M. Amig\'o, S.~Elizalde, and M.~Kennel, \emph{Forbidden patterns and shift
  systems}, J. Combin. Theory Ser. A \textbf{115} (2008), no.~3, 485--504.

\bibitem{AZS}
J.M. Amig\'o, S.~Zambrano, and M.A.F. Sanju\'an, \emph{True and false forbidden
  patterns in deterministic and random dynamics}, Europhys. Lett. \textbf{79}
  (2007), no.~50001.

\bibitem{AZS08}
\bysame, \emph{Combinatorial detection of determinism in noisy time series},
  Europhys. Lett. \textbf{83} (2008), no.~60005.

\bibitem{AZS2}
\bysame, \emph{Detecting determinism in time series with ordinal patterns: a
  comparative study}, Internat. J. Bifur. Chaos Appl. Sci. Engrg. \textbf{20}
  (2010), no.~9, 2915--2924.

\bibitem{ArcAllow}
K.~Archer, \emph{Characterization of the allowed patterns of signed shifts},
  Submitted. arXiv:1506.03464.

\bibitem{AE}
K.~Archer and S.~Elizalde, \emph{Cyclic permutations realized by signed
  shifts}, J. Comb. \textbf{5} (2014), no.~1, 1--30.

\bibitem{BKP}
C.~Bandt, G.~Keller, and B.~Pompe, \emph{Entropy of interval maps via
  permutations}, Nonlinearity \textbf{15} (2002), no.~5, 1595--1602.

\bibitem{Elishifts}
S.~Elizalde, \emph{The number of permutations realized by a shift}, SIAM J.
  Discrete Math \textbf{23} (2009), 765--786.

\bibitem{Elicyc}
\bysame, \emph{Descent sets of cyclic permutations}, Adv. in Appl. Math.
  \textbf{47} (2011), no.~4, 688--709.

\bibitem{Elibeta}
\bysame, \emph{Permutations and $beta$-shifts}, J. Combin. Theory Ser. A
  \textbf{118} (2011), no.~8, 2474--2497.

\bibitem{Eliu}
S.~Elizalde and Y.~Liu, \emph{On basic forbidden patterns of functions},
  Discrete Appl. Math. \textbf{159} (2011), no.~12, 1207--1216.

\bibitem{EM}
S.~Elizalde and K.~Moore, \emph{Patterns of negative shifts and beta-shifts},
  arxiv:1512.04479.

\bibitem{lalonde}
Scott~M. LaLonde, \emph{A computational approach to measuring homeomorphic
  defect}, Master's thesis, Clarkson University, Potsdam, NY, May 2009.

\bibitem{Makarov}
M.~Makarov, \emph{On permutations generated by infinite binary words}, Sib.
  Elektron. Mat. Izv. \textbf{3} (2006), 304--311.

\bibitem{bollt-skufca}
Joseph~D. Skufca and Erik~M. Bollt, \emph{A concept of homeomorphic defect for
  defining mostly conjugate dynamical systems}, Chaos \textbf{18} (2008),
  no.~1.

\bibitem{stuart-humph}
Andrew Stuart and A.~R. Humphries, \emph{Dynamical systems and numerical
  analysis}, Cambridge Monographs on Applied and Computational Mathematics,
  vol.~2, Cambridge University Press, Cambridge, 1998.

\bibitem{zheng}
Jiongxuan Zheng, Joseph~D. Skufca, and Erik~M. Bolt, \emph{Regularity of
  commuter functions for homeomorphic defect measure in dynamical systems model
  comparison}, Dyn. Contin. Discrete Impuls. Syst. Ser. A Math. Anal.
  \textbf{18} (2011), no.~3, 363--382.

\end{thebibliography}

\end{document}